\documentclass[11pt]{amsart}
\usepackage{amsfonts,amsmath,amssymb,amsthm}
\usepackage{latexsym}
\usepackage{eucal}
\usepackage[ansinew]{inputenc}
\usepackage[american]{babel}
\usepackage{amsfonts}
\usepackage{amssymb}
\usepackage[dvips]{graphicx}
\newtheorem{teo}{Theorem}[section]
\newtheorem{lema}{Lemma}[section]

\newtheorem{obs}{Remark}[section]

\begin{document}

\title[Uniform Decay Rates for the Damped Korteweg-de Vries.]
{Exponential Decay Rates for the Damped Korteweg-de Vries Type
Equation.}

\author[M. M. Cavalcanti]{Marcelo M. Cavalcanti $\dagger$}
\author[V. N. Domingos Cavalcanti]{Valéria N. Domingos Cavalcanti $\dagger$}
\author[F. Natali]{F\'abio M. A. Natali $\dagger$}


\keywords{Korteweg-de Vries, exponential decay.}
\thanks{{\it Date}: October, 2008.}

\maketitle

{\scriptsize \centerline{$\dagger$ Department of Mathematics -
State University of Maringá}
 \centerline{Avenida Colombo, 5790, CEP 87020-900,
 Maring\'a, PR, Brazil.}}

\begin{abstract}

The exponential decay rate of $L^2-$norm related to the
Korteweg-de Vries equation with localized damping posed on the
whole real line is established. In addition, using classical
arguments we determine that the solutions associated to the fully
damped Korteweg-de Vries equation do not decay in $H^1-$ level for
arbitrary initial data.
\end{abstract}

\section{Introduction}
The present paper sheds new light on the study of exponential decay
rate of the energy associated with mild solutions of the nonlinear
damped Korteweg-de Vries type equation given by
\begin{equation}\label{kdv1}
\displaystyle\left\{\begin{array}{lll} u_t+uu_x+u_{xxx}+a(x)u=0,\ \
\ \ \ \ \
(x,t)\in\mathbb{R}\times[0,+\infty),\\\\
u(x,0)=u_0(x),\ \ \ \ \ \ \ \ \ \ \ \ \ \ \ \ \ \ \ \ \ \ \ \ \
x\in\mathbb{R},\end{array}\right.
\end{equation}
with the following assumptions:\\

$(H1)$ ~ $a\in W^{2,\infty}(\mathbb{R})$
is a nonnegative function in $\mathbb{R}$, satisfying
$a(x)\geq\alpha_0>0\ \mbox{for}\ x>R_1,\ R_1>0$ (or
$a(x)\geq\lambda_0>0\ \mbox{for}\ x< -R_2,\ R_2>0$).\\

The function $a(x)$ presented in equation $(\ref{kdv1})$ is
responsible for the localized effect of the dissipative mechanism.\\
\indent Equation $(\ref{kdv1})$ is a generalization of the well-known Korteweg-de Vries equation, that is
\begin{equation}\label{kdvorig}
u_t+uu_x+u_{xxx}=0,
\end{equation}
and was established by Boussinesq in $1877$, and later, in $1895$,
Korteweg and de Vries proved that this equation is an approximate
description of surface water waves propagating in a canal. An
important characteristic of this equation is its travelling-wave
solutions, that is, its special solutions of the form
$u(x,t)=\varphi(x-ct)$, where $\varphi$ denotes the \textit{wave
profile} and $c>0$ is the \textit{wave speed}. The study of the
nonlinear orbital stability or instability for equation
(\ref{kdvorig}) based on travelling waves has had a considerable
development and refinement in recent years (e.g \cite{albert1},
\cite{natali}, \cite{angulo}, \cite{benjamin1}, \cite{bona1}, and
\cite{W1} and references therein).\\

\indent In particular, when it is considered in equation
$(\ref{kdv1})$, $a(x)=\mu>0$, where $\mu$ is a sufficiently small
positive constant, the stability of the solitary wave can be
studied using asymptotic methods which were developed for
perturbed solitons (see \cite{grimshaw}). The variation of the
solitary wave amplitude (weak-amplitude) is given by the energy
balance equation:
\begin{equation}\label{E0}
\displaystyle\frac{dE_0(t)}{dt}=-2\mu\int_{\mathbb{R}}u^2dx,
\end{equation}
where $E_0(t)=||u(t)||_{L^2(\mathbb{R})}^2$ is the momentum
(energy), substituting the expression for a solitary wave
directly into $(\ref{E0})$.\\

\indent Equation (\ref{kdv1}) has attracted considerable
attention, particularly when the exponential decay rates are
studied for equations posed on a finite interval $(0,L)$, instead
of the whole line; for instance see \cite{Bisognin},
\cite{Felipe}, \cite{KDV-weak}, \cite{Pazoto}, \cite{Perla} and
references therein. According to our best knowledge, to deal with
dissipative effects in the {\em whole line} and assuming that the
function $a(\cdot)$ {\em does not possess} compact support, as in
the present paper, has never been treated so far in the
literature. This brings up new difficulties when establishing the
uniform stabilization of the
energy, namely, $E_0(t):=||u(t)||_{L^2(0,L)}^2$,~$t\geq 0$.\\

\indent To overcome these difficulties, we have to combine {\em
sharp estimates} due to Kenig, Ponce and Vega \cite{KPV91} in
order to establish the well-posedness of problem (\ref{kdv1}),
with unique continuation properties proved in Zhang \cite{zhang}
and new tools, in order to derive exponential and uniform decay
rates of the energy. It is important to observe that the majority
of the papers in the literature regarding this subject (decay
rates estimates for $E_0$), deals with regular solutions in order
to employ unique continuation properties, for instance see Kenig,
Ponce and Vega \cite{KPV91}, Zhang \cite{zhang} and references
therein. In the present case, using the continuous dependence and
arguments of density we are going to work with smooth solutions
which enable us to use arguments of unique continuation in order
to obtain the desired decay rate estimates for mild solutions.

\medskip

We observe that a mild solution $u$ of (\ref{kdv1}) belongs to
$C_B([0,T];H^1(\mathbb{R}))$, then it is natural to look for
exponential stability in $H^1-$level instead of $L^2-$level; that
is, it is natural to investigate if the energy of first order
$E_1(t)=\frac12||u(t)||_{H^1}^2-\frac13\int_{\mathbb{R}}u^3\,dx$,
or at least, the $H^1-$norm: $||u(t)||_{H^1}^2$, decays
exponentially. But, even if $E_1'(t)\leq 0$, for all $t \geq 0$,
the exponential decay does not hold for arbitrary mild solutions,
according to section 3.  The last statement  is a new approach in the context of dispersive equation and
it is proved using the well- known {\em potential well} theory,
see for instance references
\cite{Cavalcanti}, \cite{Cavalcanti2}, \cite{Cavalcanti3} or \cite{vitillaro}.\\

\medskip

In order to show the current findings, this paper is organized as
follows. Section 2 establishes the well-posedness  and exponential
decay of the energy related to equation $(\ref{kdv1})$. Section 3
is devoted to prove that the decay rate estimates in $H^1-level$
is not expected for arbitrary initial data.

\section{Well-Posedness and Exponential Decay for Equation $(\ref{kdv1})$.}
\subsection{Well-Posedness.}
\hspace{0.2cm}The usual spaces $L^2(\mathbb{R})$ and
$H^1(\mathbb{R})$ will be considered endowed with the followings
norms
$$||u||_{L^2}:=||u||_{L^2(\mathbb{R})}=\displaystyle\left(\int_{\mathbb{R}}
|u(x)|^2dx\right)^{1/2},$$
and
$$||u||_{H^1}:=||u||_{H^1(\mathbb{R})}=\displaystyle\left(||u||_{L^2}^2+
||u_x||_{L^2}^2\right)^{1/2}.$$

 For $1\leq p<\infty$ we define the
norm in $L^p(\mathbb{R})$
$$||u||_{L^p}:=||u||_{L^p(\mathbb{R})}=
\displaystyle\left(\int_{\mathbb{R}}|u(x)|^pdx\right)^{1/p}$$ and
for $p=\infty$ we define
$$||u||_{L^{\infty}}:=||u||_{L^{\infty}(\mathbb{R})}=
\displaystyle\mbox{ess}\sup_{x\in\mathbb{R}}|u(x)|.$$

 Next, we
shall consider similar sets introduced by Kenig, Ponce and Vega in
\cite{KPV91}.  For each $T>0$ and each measurable function
$u:\mathbb{R}\times[-T,T]\rightarrow\mathbb{R}$ we define

\begin{equation}
\gamma_1(u,T)=\displaystyle
\mbox{ess}\sup_{t\in[-T,T]}||u(t)||_{H^1(\mathbb{R})},\label{gamma1}
\end{equation}

\begin{equation}
\gamma_2(u,T)=\displaystyle\left(
\int_{-T}^{T}||u_{xx}(t)||_{L^{\infty}(\mathbb{R})}^2dt\right)^{1/2}\label{gamma2},
\end{equation}

\begin{equation}
\gamma_3(u,T)=\displaystyle
\left(\int_{-T}^{T}||u_x(t)||_{L^{\infty}(\mathbb{R})}^6dt\right)^{1/6},\label{gamma3}
\end{equation}

\begin{equation}
\gamma_4(u,T)=\displaystyle\frac{1}{1+T}\left(
\mbox{ess}\sup_{t\in[-T,T]}||u(t)||_{L^2(\mathbb{R})}^2dx\right)^{1/2}
\label{gamma4}
\end{equation}
and

\begin{equation}
\Gamma(u,T)=\max_{i=1,2,3,4}\gamma_i(u,T).\label{Gamma}
\end{equation}

Let us define
\begin{eqnarray}\label{X_T}
X_T=\{u\in C([-T,T], H^1(\mathbb{R})), \Gamma(u,T)< +\infty\},
\end{eqnarray}
and
\begin{equation}\label{XT}
X_{T}^k=\{u\in C([-T,T], H^1(\mathbb{R})),\
\Gamma(u,T)\leq\kappa<+\infty\},\end{equation} where $\kappa>0$
will be
chosen later.\\

It is straightforward to prove that $X_T$ defined in $(\ref{XT})$
is a Banach space when endowed with the norm
$||\cdot||_{X_{T}}:=\Gamma(\cdot,T)$.

\subsection{Linear Estimates.} In this subsection the
estimates related to the linear Korteweg-de Vries equation is
considered, that is, estimates related to problem
\begin{equation}\left\{\begin{array}{lll}
v_t+v_{xxx}=0,\ \ \ (x,t)\in\mathbb{R}\times\mathbb{R}\\\\
v(x,0)=v_0(x), \ \ x\in\mathbb{R}. \label{linearkdv}
\end{array}\right.
\end{equation}

The proof of the following result can be found in Kenig, Ponce and
Vega  \cite{KPV91}.

\begin{lema} Let us denote by $\{S(t)\}_{t\in\mathbb{R}}$ the group associated
with the linear equation $(\ref{linearkdv})$.  There is a constant
$c_1>1$ such that
\begin{equation}\label{linearest1}
\displaystyle
\left(\int_{-\infty}^{+\infty}||S(t)v_0||_{L^{\infty}}^6dt\right)^{1/6}\leq
c_1||v_0||_{L^2},
\end{equation}
for all $v_0\in L^2(\mathbb{R})$, and
\begin{equation}\label{linearest2}
\displaystyle\left(\int_{\mathbb{R}}\
\mbox{ess}\sup_{t\in[-T,T]}|S(t)v_0(x)|^2dx\right)^{1/2} \leq
c_1(1+T)||v_0||_{H^1},
\end{equation}
for all $v_0\in H^1(\mathbb{R})$ and $T>0$. Moreover, for all
$v_0\in L^2(\mathbb{R})$ we have,
\begin{equation}
\label{linearest3} \displaystyle\left( \mbox{ess}\
\sup_{x\in\mathbb{R}}
\int_{-T}^{T}|\partial_xS(t)v_0(x)|^2dt\right)^{1/2}\leq
c_1||v_0||_{L^2}.
\end{equation}
\label{lema1}
\end{lema}

Since $S(t):L^2(\mathbb{R})\rightarrow L^2(\mathbb{R})$ and
$S(t):H^1(\mathbb{R})\rightarrow H^1(\mathbb{R})$ are isometries
for all $t\in\mathbb{R}$, if the initial condition $v_0$ belongs
to $H^1(\mathbb{R})$, the solution associated to
$(\ref{linearkdv})$ belongs to $C_{B}([-T,T];H^1(\mathbb{R}))$,
that is,
\begin{equation}\label{contlinearkdv}
v_0\in H^1(\mathbb{R})\mapsto v\in C_{B}([-T,T];H^1(\mathbb{R})),
\end{equation}
for all $T>0$, where $v(t)=S(t)v_0$.

From Lemma $\ref{lema1}$ and $(\ref{contlinearkdv})$ we are able to
present the following result,

\begin{lema}\label{lema2}
Let $T>0$ and $v_0\in H^1(\mathbb{R})$. Defining
$v(\cdot,t)=S(t)v_0$, we obtain that $v\in X_T$ and,
\begin{equation}\label{estXTlinear}
||v||_{X_{T}}\leq c_1||v_0||_{H^1},
\end{equation}
where $c_1>1$ is the constant given in Lemma $\ref{lema1}$ and
does not depend on $T>0$ and $v_0$.
\end{lema}
\begin{proof}
Indeed, since $w\in H^1(\mathbb{R})\mapsto S(t)w\in
C_{B}([-T,T];H^1(\mathbb{R}))$ is a linear isometry, we get from
$(\ref{contlinearkdv})$ that $\gamma_1(v,T)\leq c_1||v_0||_{H^1}$.\\

\indent From  $(\ref{linearest3})$ we deduce
\begin{equation}\label{est1}\begin{array}{lll}
\gamma_2(v,T)&=&\displaystyle\left(
\mbox{ess}\sup_{x\in\mathbb{R}}\int_{-T}^{T}
|\partial_x\left(S(t)\partial_xv_0(x,t)\right)|^2dt\right)^{1/2}\\\\
&\leq&\displaystyle c_1 ||\partial_xv_0||_{L^2} \leq
c_1||v_0||_{H^1}.
\end{array}
\end{equation}

Using $(\ref{linearest1})$ and $(\ref{linearest2})$ we conclude
$$\gamma_3(v,T)\leq c_1||v_0||_{H^1}\ \ \ \ \mbox{and}\ \ \
\gamma_4(v,T)\leq c_1||v_0||_{H^1}.$$

This arguments establish the lemma.
\end{proof}

The next step is to analyze the non-homogenous equation associated
with $(\ref{linearkdv})$, given by
\begin{equation}\left\{\begin{array}{lll}
\varphi_t+\varphi_{xxx}=g,\ \ \ (x,t)\in\mathbb{R}\times\mathbb{R}\\\\
\varphi(x,0)=\varphi_0(x), \ \ x\in\mathbb{R}. \label{nonhomkdv}
\end{array}\right.
\end{equation}

In Kenig, Ponce and Vega \cite{KPV91} they proved the following
result:

\begin{lema}\label{lema3}
Let $T>0$ and $g\in L^1([-T,T];H^1(\mathbb{R}))$. If we define
\begin{equation}\label{eqintg}
\varphi(\cdot,t)=\int_{0}^{t}S(t-\tau)g(\cdot,\tau)d\tau,
\end{equation}
for $t\in[0,T]$, then $\varphi$ belongs to $X_T$ and
\begin{equation}\label{estXTnonhom}
||\varphi||_{X_T}\leq c_1||g||_{L^1([-T,T];H^1(\mathbb{R}))},
\end{equation}
where $c_1>1$ is the constant given in Lemma $\ref{lema1}$.\\
\end{lema}

\begin{proof} We denote by $\psi^t(\tau)$ the characteristic
function over the interval $[0,t]$ for $0\leq t<T$. Therefore, from
$(\ref{eqintg})$ we can write

\begin{equation}\label{eqint1}
\varphi(\cdot,t)=\displaystyle\int_{-T}^{T}\psi^t(\tau)S(t-\tau)g(\cdot,\tau)d\tau.
\end{equation}

Then, from $(\ref{linearest3})$ we conclude that
\begin{equation}\label{est2}\begin{array}{lll}
\gamma_2(\varphi,T)&\leq& \displaystyle\int_{-T}^{T}\left(ess
\sup_{x\in\mathbb{R}}\int_{-T}^{T}[\partial_x^2
(\psi^t(\tau)S(t-\tau)g(\cdot,\tau))]^2dt\right)^{\frac{1}{2}}d\tau
\\\\
&\leq&\displaystyle\int_{-T}^{T}\left(ess
\sup_{x\in\mathbb{R}}\int_{-T}^{T}[\partial_x((\tau)S(t')
\partial_xg(\cdot,\tau))]^2dt'\right)^{\frac{1}{2}}d\tau
\\\\
&\leq&\displaystyle\int_{-T}^{T}||g(\cdot,\tau)||_{H^1}d\tau
=\frac{\sqrt{3}}{3}||g||_{L^1([-T,T];H^1(\mathbb{R}))}
\end{array}\end{equation}

 On the other hand, since $S(t):H^1(\mathbb{R})\rightarrow
H^1(\mathbb{R})$ is an isometry, we deduce from Bochner Theorem
that

\begin{equation}\label{est3}
\begin{array}{lll}
||\varphi(t)||_{H^1(\mathbb{R})}&=&\displaystyle\left\|\int_{-T}^{T}\psi^t(\tau)S(t-\tau)
g(\cdot,\tau)d\tau\right\|_{H^1(\mathbb{R})}\\\\
&\leq&\displaystyle\int_{-T}^{T}||S(t-\tau)g(\cdot,\tau)||_{H^1(\mathbb{R})}d\tau\leq
c_1||g||_{L^1([-T,T];H^1(\mathbb{R}))}.
\end{array}
\end{equation}
Thus $\gamma_1(\varphi,T)\leq
c_1||g||_{L^1([-T,T];H^1(\mathbb{R}))}$. Considering similar ideas
used to prove $(\ref{est2})$, we deduce from $(\ref{linearest1})$
and $(\ref{linearest3})$ that
$$\gamma_i(\varphi,T)\leq c_1||g||_{L^1([-T,T];H^1(\mathbb{R}))},$$ where $i=3,4$.
\end{proof}

We turn our attention to the non-linear term $uu_x$ present in
equation (1.1).  In fact, from \cite{KPV91} we have the following
statement:

\begin{lema}\label{lema4}
Let $T>0$ and let $u$, $w\in X_T$. Then,
\begin{equation}\label{bilinearest1}
||(uw)_x||_{L^1([-T,T];H^1(\mathbb{R}))}\leq
4(1+\sqrt{2})T^{1/2}(1+T)||u||_{X_{T}}||w||_{X_{T}}.
\end{equation}
\end{lema}

 As a consequence of Lemmas $\ref{lema3}$ and $\ref{lema4}$ if we
 define
\begin{equation}\label{eqintbilinear}
\varphi(\cdot,t)=\int_{0}^{t}S(t-\tau)(uw)_xd\tau,
\end{equation}
for $t\in [-T,T],T>0$ and $u,v \in X_T$, we obtain that

\begin{equation}\label{bilinearest2}
||\varphi||_{X_{T}}\leq
4(1+\sqrt{2})c_1T^{1/2}(1+T)||u||_{X_{T}}||w||_{X_{T}}.
\end{equation}

In order to obtain an analogous estimate as in
$(\ref{estXTnonhom})$, it is necessary to analyze the term
$a(\cdot)u$. In fact,
\begin{equation}\label{L2ax}
\displaystyle\int_{-T}^{T}||a(\cdot)u(t)||_{L^2}dt
\leq||a||_{W^{2,\infty}}\int_{-T}^{T}||u(t)||_{L^2}dt
\leq\sqrt{2}||a||_{W^{2,\infty}}T||u||_{X_T},
\end{equation}
and
\begin{equation}\label{H1ax}\begin{array}{lll}
\displaystyle\int_{-T}^{T}||(a(\cdot)u(t))_x||_{L^2}dt
&\leq&\displaystyle\sqrt{2}\int_{-T}^{T}||a_xu(t)||_{L^2}dt
+\sqrt{2}\int_{-T}^{T}||au_x(t)||_{L^2}dt\\\\
&\leq&\sqrt{2}||a||_{W^{2,\infty}}T||u||_{X_T}.
\end{array}\end{equation}
Therefore,
\begin{equation}\label{estax}
||a(\cdot)u||_{L^1([-T,T];H^1(\mathbb{R}))}\leq2\sqrt{2}||a||_{W^{2,\infty}}T||u||_{X_T}.
\end{equation}

From $(\ref{estax})$ we obtain the following result:

\begin{lema}\label{lema5}
Let $T>0$ and $u\in X_T$. Consider the integral equation given by
\begin{equation}\label{eqintax}
\varphi(\cdot,t)=\int_{0}^{t}S(t-\tau)(a(\cdot)u(\tau))d\tau,\ \ \ \
t\in[0,T].
\end{equation}
Then
\begin{equation}\label{esteqintax}
||\varphi||_{X_T}\leq 2\sqrt{2}c_1||a||_{W^{2,\infty}}T||u||_{X_T},
\end{equation}
where $c_1>1$ is the constant given in Lemma $\ref{lema1}$.
\end{lema}
\begin{flushright}
${\square}$
\end{flushright}

\subsection{Local Well-Posedness.} Let us consider an initial data
$u_0\in H^1(\mathbb{R})$. The aim in this subsection is to find a
local solution of $(\ref{kdv1})$ in the mild sense (see
\cite{pazy}), that is, to find the unique fixed point of the map
$\Psi:X_T\rightarrow X_T$ defined by
\begin{equation}
\displaystyle\Psi(u)(t)=S(t)u_0-\int_{0}^{t}S(t-\tau)\left(a(\cdot)u(\tau)
+\frac{1}{2}(u(\tau)^2)_x\right)d\tau,\label{mildintegral}
\end{equation}
for $t\in[-T,T]$. Indeed, first of all it is necessary to show
that $\Psi$  is well-defined. Let $u\in X_T$, then from Lemmas
$\ref{lema4}$ and $\ref{lema5}$ we get
\begin{equation}\begin{array}{lll}
||\Psi(u)(t)||_{X_T}&=&\displaystyle\left\|S(t)u_0+
\int_{0}^{t}S(t-\tau)\left(-a(\cdot)u(\tau)-
\frac{1}{2}(u(\tau)^2)_xd\tau\right)\right\|_{X_{T}}\\\\
&\leq& \displaystyle
c_1||u_0||_{X_{T}}+2\sqrt{2}c_1||a||_{W^{2,\infty}}T||u||_{X_T}+
\frac{c_2}{2}T^{1/2}(1+T)||u||_{X_{T}}^2,
\end{array}\label{estpsi1}\end{equation}
where $c_2=4(1+\sqrt{2})c_1$. Since $u\in X_{T}$ we have for
$\kappa=2c_1||u_0||_{H^1}$ that
\begin{equation}\label{estkappa}
||\Psi(u)(t)||_{X_T}\leq
c_1||u_0||_{H^1}\left(1+4\sqrt{2}c_1||a||_{W^{2,\infty}}T+
2c_1c_2T^{1/2}(1+T)||u_0||_{H^1}\right).
\end{equation}

Taking $0<T_{\kappa}<1$ verifying
\begin{equation}\label{estTkappa}
4\sqrt{2}c_1||a||_{W^{2,\infty}}T_{\kappa}+
2c_1c_2T_{\kappa}^{1/2}(1+T_{\kappa})||u_0||_{H^1}<1,
\end{equation}
it is possible to conclude
$\displaystyle||\Psi(u)(t)||_{X_T}<2c_1||u_0||_{H^1}=\kappa$. \\
\indent Next, we prove that $\Psi$ in $(\ref{mildintegral})$ is a
strict contraction. In fact, let $u$, $w\in X_T$; then
\begin{equation}\begin{array}{lll}
||\Psi(u)-\Psi(w)||_{X_T}
&\leq&\displaystyle\left\|\int_{0}^{t}S(t-\tau)(a(\cdot)(u-w)(\tau)d\tau\right\|_{X_T}+\\\\
&+&
\displaystyle\frac{1}{2}\left\|\int_{0}^{t}S(t-\tau)(u(\tau)^2-w(\tau)^2)_xd\tau\right\|_{X_T}.
\end{array}
\label{contraction1}
\end{equation}

From the equality  $(u-w)^2=(u-w)(u+w)$, Lemmas $\ref{lema4}$ and
$\ref{lema5}$ and since $\kappa=2c_1||u_0||_{H^1}$ we obtain,
\begin{equation}
\begin{array}{lll}
||\Psi(u)-\Psi(w)||_{X_T} &\leq&\displaystyle
2\sqrt{2}c_1||a||_{W^{2,\infty}}T||u-w||_{X_T}+
\frac{c_2}{2}T^{1/2}(1+T)||u-w||_{X_{T}}||u+w||_{X_{T}}\\\\
&\leq&\displaystyle\left(2\sqrt{2}c_1||a||_{W^{2,\infty}}T+2c_1c_2
T^{1/2}(1+T)||u_0||_{H^1}\right)||u-w||_{X_{T}}.
\end{array}
\label{contraction2}
\end{equation}

Since $0<T_{\kappa}<1$ was chosen verifying
$$4\sqrt{2}c_1||a||_{W^{2,\infty}}T_{\kappa}+
2c_1c_2T_{\kappa}^{1/2}(1+T_{\kappa})||u_0||_{H^1}<1,$$ we
guarantee the existence of $0<\alpha<1$ such that
$$||\Psi(u)-\Psi(w)||_{X_T}<\alpha||u-w||_{X_T}.$$

Then, $\Psi: X_{T_{\kappa}}\rightarrow X_{T_{\kappa}}$ is a strict
contraction and therefore, from Banach fixed point Theorem, it is
possible to conclude that problem $(\ref{kdv1})$ possesses a
unique
(mild) solution.\\
\indent This argument proves the following Theorem.

\begin{teo}
Suppose that $a\in W^{2,\infty}(\mathbb{R})$ satisfies the
assumption $(H1)$ and consider that $u_0\in H^1(\mathbb{R})$.
Given $\kappa>0$ and $0<T_{\kappa}<1$ defined in
$(\ref{estTkappa})$; then, there is a unique mild solution $u\in
X_{T_{\kappa}}$ of equation $(\ref{kdv1})$. Moreover, the map
$$u_0\in
H^1(\mathbb{R})\mapsto u\in C_B([-T',T']; H^1(\mathbb{R})),$$ is
continuous for an appropriate $0<T'<T_{\kappa}<1$.
\label{teo1}\end{teo}
\begin{flushright}
${\square}$
\end{flushright}
\begin{obs}
It is possible to deduce, taking the duality product in  equation
(\ref{kdv1}) that if $u$ is a solution in the mild sense, then $u$
is also a solution in the weak sense, and vice-versa, the weak
solution which belongs to $X_T$ and is the fixed point of $\Psi$,
is the unique mild solution. \label{obs1}\end{obs}

\subsection{Global Well-Posedness.} A result of global well-posedness associated to equation
$(\ref{kdv1})$ for given initial data $u_0\in H^1(\mathbb{R})$,
will be established in this subsection. Indeed, from local
well-posedness result (see Theorem $\ref{teo1}$), it is possible
to obtain some estimates in $H^1(\mathbb{R})$ in order to get the
required result. Before going on, we restrict ourselves on mild
solutions, obtained above, such that $u\in
C_B([0,T'];H^1(\mathbb{R}))$ instead of $u\in
C_B([-T',T'];H^1(\mathbb{R}))$, for an appropriate
$0<T'<T_{\kappa}<1$ (see Theorem $\ref{teo1}$) to obtain the
exponential decay in $L^2-$norm (which we shall present in the
next section).\\
\indent Multiplying $(\ref{kdv1})$ by $u$ and integrating over
$\mathbb{R}\times (0,T')$, where $0<t<T'$ and $T'>0$ is given in
Theorem $\ref{teo1}$, we get

\begin{equation}\label{estI}
\displaystyle\int_{\mathbb{R}}|u|^2dx=-2\int_{0}^{t}\int_{\mathbb{R}}a(x)|u|^2dxds+||u_0||_{L^2}^2\leq||u_0||_{L^2}^2.
\end{equation}

Therefore we can conclude that

\begin{equation}\label{limit1}
u\ \mbox{is bounded in}\ L^{\infty}([0,+\infty);L^2(\mathbb{R})).
\end{equation}

Next, multiplying the first equation in $(\ref{kdv1})$ by
$\displaystyle u^2+2u_{xx}$ and integrating in $x\in\mathbb{R}$ we
have the following equality:

\begin{equation}
\displaystyle\frac{d}{dt}\int_{\mathbb{R}}\left[\frac{u^3}{3}-u_x^2\right]dx+\int_{\mathbb{R}}a(x)u^3dx+2\int_{\mathbb{R}}a(x)uu_{xx}dx=0.
\label{conservada1}
\end{equation}

But
$$\displaystyle\int_{\mathbb{R}}a(x)uu_{xx}dx=\frac{1}{2}\int_{\mathbb{R}}a_{xx}(x)u^2dx-\int_{\mathbb{R}}a(x)u_x^2dx,$$
and therefore, from $(\ref{conservada1})$ we obtain after
integrating once in $t\in [0,T')$ that

\begin{equation}\begin{array}{lll}
\displaystyle\int_{\mathbb{R}}\left[u_x^2-\frac{u^3}{3}\right]dx
&-&\displaystyle\int_{0}^{t}\int_{\mathbb{R}}a(x)u^3dxds-\int_{0}^{t}
\int_{\mathbb{R}}a_{xx}(x)u^2dxds\\\\
&+&\displaystyle 2\int_{0}^{t}\int_{\mathbb{R}}a(x)u_x^2dxds=
\int_{\mathbb{R}}\left[u_{0_{x}}^2-\frac{u_0^3}{3}\right]dx.
\end{array}\label{conservada2}\end{equation}

Combining (\ref{estI}) and (\ref{conservada2}) and using
Gagliardo-Nirenberg and Young inequalities we deduce that

\begin{equation}
||u(t)||_{H^1}^2\leq (3||a||_{W^{2,\infty}}+ \frac 1 2 +C_{\frac 1
2}||u_{0}||_{H^1}^{\frac 4 3}+
||u_{0}||_{H^1})\int_{0}^{t}||u(s)||_{H^1}^2ds+C_5(||u_{0}||_{H^1}^{2}+||u_{0}||_{L^2}^{\frac
{10}3}+ ||u_{0}||_{L^2}^{3}). \label{estH1}\end{equation}

 From Gronwall
inequality we have,

\begin{equation}
||u(t)||_{H^1}^2\leq
C_5(||u_{0}||_{H^1}^{2}+||u_{0}||_{L^2}^{\frac {10} 3}+
||u_{0}||_{L^2}^{3})e^{(3||a||_{W^{2,\infty}}+ \frac 1 2 +C_{\frac
1 2}||u_{0}||_{H^1}^{\frac 4 3})T}. \label{gronwall}\end{equation}

Inequality $(\ref{gronwall})$ able us to enunciate the next
result,

\begin{teo}
Suppose that $a\in W^{2,\infty}(\mathbb{R})$ satisfies assumption
$(H1)$ and consider $u_0\in H^1(\mathbb{R})$. Then, there exists a
unique mild solution $u$ which belongs to
$$L_{loc}^{\infty}([0,+\infty);H^1(\mathbb{R}))\cap X_T,$$ for all
$T>0$,
and it satisfies the inequality $(\ref{estH1})$. Moreover, the map
$$u_0\in H^1(\mathbb{R})\mapsto u\in C_B([0,T]; H^1(\mathbb{R})),$$ is
continuous for all $T>0$. \label{teo2}
\end{teo}
\begin{flushright}
${\square}$
\end{flushright}

\begin{obs}\label{obs2}
The equalities obtained in $(\ref{estI})$, $(\ref{conservada1})$,
and the estimate $(\ref{gronwall})$ were deduced considering
regular solutions for equation $(\ref{kdv1})$, for instance, if we
take an initial data $\bar{u}_{0}\in C_{0}^{\infty}(\mathbb{R})$
and $a\in W^{s,\infty}(\mathbb{R})$ for $s>0$ large enough, the
local solution $\bar{u}\in X_T$ given by Theorem $\ref{teo1}$, for
some $T>0$, coincides with the classical solution, which exists
globally and belongs to
$C^{\infty}(\mathbb{R}\times{\mathbb{R}})$. By density arguments
we get identities $(\ref{estI})$ and $(\ref{conservada1})$ for a
mild solution $u\in X_T$ belonging to
$C_{B}(\mathbb{R};H^1(\mathbb{R}))$ and satisfying the inequality
$(\ref{gronwall})$.
\end{obs}

\subsection{Exponential Decay for equation $(\ref{kdv1})$.} In this
subsection we are interested in obtaining an exponential decay
rate for the energy norm in $L^2(\mathbb{R})$ related with the
Korteweg-de
Vries equation $(\ref{kdv1})$ with localized damping.\\

\indent In fact, multiplying the first equation in $(\ref{kdv1})$
by $u$ and integrating over $\mathbb{R}\times (0,t)$,
$t\in(0,+\infty)$ (see remark $\ref{obs2}$) we get

\begin{equation}
E_0(t):=\displaystyle\frac{1}{2}\int_{\mathbb{R}}|u(x,t)|^2dx=
-\int_{0}^{t}\int_{\mathbb{R}}a(x)|u(x,t)|^2dxds+\frac{1}{2}||u_0||_{L
^2}^2\leq\frac{1}{2}||u_0||_{L ^2}^2. \label{energy}\end{equation}

Since $\displaystyle\frac{d}{dt}E_0(t)\leq0$, it makes sense to
look for exponential decay rates. From equation $(\ref{kdv1})$ we
have the following energy estimate
\begin{equation}
\int_{0}^{T}E_0(t)dt\leq\frac{\alpha_0^{-1}}{2}
\left[\int_{\mathbb{R}}u^2dx\right]_{0}^{T}+\underbrace{\int_{0}^{T}\int_{x\leq
R_1}u^2dxdt.}_{A} \label{energyest}
\end{equation}

The following result establishes the exponential decay rate for
the solution $u$ obtained in Theorem $\ref{teo2}$.

\begin{teo}
Consider the potential $a(\cdot)$ satisfying hypothesis $(H1)$.
For any $L>0$, there are $c=c(L)>0$ and $\omega=\omega(L)$ such
that
$$E_0(t)\leq ce^{-\omega t},$$
for all $t\geq0$ and for any solution $u$ of $(\ref{kdv1})$
determined in Theorem $\ref{teo2}$, provided that
$||u_0||_{H^1(\mathbb{R})}\leq L$. \label{teo3}
\end{teo}

\begin{proof} Before  establishing the exponential decay
related to equation $(\ref{kdv1})$, we need to obtain a preliminar
estimate for the integral  $A$ (a similar bound can be deduced
considering the integral over $\{x\in \mathbb{R};x\geq -R_2$\}).
In fact, we have the following lemma, where $T_0$ is a positive
constant and where we consider that the initial data lies in a
bounded set of $H^1(\mathbb{R})$.

\begin{lema}
Let $u\in L_{loc}^{\infty}(0,+\infty;H^1(\mathbb{R}))\cap X_T$ be
the mild solution associated to the Korteweg-de Vries equation
$(\ref{kdv1})$, obtained from Theorem $\ref{teo2}$. Then, we have
that for all $T>T_0$ there exists a positive constant
$c_5=c_5(T_0,E_0(0))$ such that if $u$ is the mild solution of
$(\ref{kdv1})$ with initial data $u_0 \in H^1(\mathbb{R})$, the
following inequality holds
\begin{equation}
\int_{0}^{T}\int_{x\leq R_1}u^2dxdt\leq
c_5\int_{0}^{T}\int_{\mathbb{R}}a(x)u^2dxdt,
\label{estUCP}\end{equation} provided that $u_0$ belongs to a
bounded set of $H^1(\mathbb{R})$.
 \label{lema6}\end{lema}

\begin{proof} Consider $P_{R_1}:=\{x\in\mathbb{R};\ x\leq R_1,\ R_1>0\}$.
 First of all we note that according to the continuous dependence of the initial data and arguments
of density it is sufficient to establish the Lemma for smooth
solutions (for instance, those ones which belong to
$C([0,T],H^s(\mathbb{R}))$), $s>0$ large enough, related to
problem
$(\ref{kdv1})$ (see Remark $\ref{obs2}$).\\
\indent We argue by contradiction. Let us suppose that
$(\ref{estUCP})$ is not verified and let
$\{u_k(0)\}_{k\in\mathbb{N}}=\{u_k^0\}_{k\in\mathbb{N}}$ be a
bounded sequence of initial data in $H^1-$norm, where the
corresponding solutions $\{u_k\}_{k\in \mathbb{N}}$ of
$(\ref{kdv1})$ verifies
\begin{equation}\label{blowUCP}
\lim_{k \rightarrow +\infty}\displaystyle\frac{\int_0^T
||u_k(t)||_{L^2(P_{R_1})}^2 dt}{\int_0^T\int_{\mathbb{R}}\left(
a(x)\,u_k^2\right)dx\,dt}=+\infty,
\end{equation}
that is,
\begin{equation}\label{inverseUCP}
\displaystyle\lim_{k \rightarrow
+\infty}\frac{\int_0^T\int_{\mathbb{R}}\left(
a(x)\,u_k^2\right)dx\,dt}{\int_0^T ||u_k(t)||_{L^2(P_{R_1})}^2 dt}=0.
\end{equation}

Since,
$$E_{0}^k(t)\leq E_{0}^k(0)\leq L,$$
we obtain a subsequence of $\{u_k\}_{k\in\mathbb{N}}$, still denoted
by $\{u_k\}_{k\in\mathbb{N}}$ from now on, which verifies the
convergence:
\begin{equation}
u_k\rightharpoonup u\ \ \ \mbox{ weakly in }\ \ \ \
L^{\infty}([0,T];L^2(\mathbb{R}))\label{convergence1}\end{equation}

From the boundedness of $\{u_k\}_{k\in\mathbb{N}}$ in
$L^{\infty}(0,T;L^2(\mathbb{R}))$ and $(\ref{blowUCP})$ we deduce

\begin{equation}
\displaystyle\lim_{k\rightarrow+\infty}\int_{0}^{T}\int_{\mathbb{R}}a(x)u_{k}^2dxdt=0,
\label{limit2}\end{equation} consequently, from the hypothesis
made on $a(x)$ we have
\begin{equation}
\displaystyle\lim_{k\rightarrow+\infty}\int_{0}^{T}\int_{\mathbb{R}\backslash
P_{R_1}}u_{k}^2dxdt=0. \label{limit3}\end{equation}

Let us consider $R\in\mathbb{R}$ such that $R<R_1$. Employing
$(\ref{gronwall})$ and compactness results we obtain

\begin{equation}
u_k\rightarrow u_{R} \ \ \ \ \mbox{strongly in}\ \ \ \
L^{2}([0,T];L^2([R,R_1]));\ \ \ \forall\ R<R_1.
\label{limit4}\end{equation}

Therefore, from $(\ref{convergence1})$ and $(\ref{limit4})$ we
obtain, from the uniqueness of weak limit, that $u_R=u$ in
$L^{2}([0,T];L^2([R,R_1]))$ for all $R<R_1$ and then, we infer,

\begin{equation}
u_k\rightarrow u \ \ \ \ \mbox{strongly in}\ \ \ \
L^{2}([0,T];L^2([R,R_1]));\ \ \ \forall\ R<R_1,
\label{limit4'}\end{equation}

Having in mind $(\ref{limit3})$ and $(\ref{limit4'})$, we conclude
that

\begin{equation}
u_k\rightarrow \tilde{u} \ \ \ \ \mbox{strongly in}\ \ \ \
L^{2}([0,T];L^2(R,+\infty)); \ \ \ \forall\
R<R_1,\label{limit5}\end{equation}
 where

\begin{equation}\label{utilde}
\tilde{u}=\left\{
\begin{array}{lll}
{u,\ \ \ \mbox{a.e. in}\ \ \ \ [R,R_1],~\forall R< R_1}\\
{0,\ \ \ \mbox{a.e. in}\ \ \ \ \mathbb{R}\backslash P_{R_1}}.
\end{array}
\right.
\end{equation}

In addition, from $(\ref{estH1})$ it is possible to conclude that,
\begin{equation}
u_{k,x}\rightharpoonup u_x\ \ \ \mbox{ weakly in }\ \ \ \
L^{2}([0,T];L^2(\mathbb{R}))\label{convergence2}\end{equation}

At this point we will divide our proof into two cases, namely: when
$u\ne 0$ and $u=0$.

\medskip

Case (I): $u \ne 0.$

In this case, passing to the limit in the equation
\begin{equation}
u_{k,t} + u_k u_{k,x} + u_{k,xxx}+ a(x) u_k=0~\hbox{ a. e. in  }
(R,+\infty) \times (0,T) ~(\forall R< R_1),
\end{equation}
when $k
\rightarrow +\infty$, we deduce that
\begin{equation}\label{kdvlimit}
\displaystyle\left\{\begin{array}{lll} u_{t}+uu_{x}+u_{xxx}=0,\ \ \
\ \ \ \ \ \ \ \
\mbox{in}\ \ C([0,T];L^2(R,+\infty)),\\\\
u=0,\ \ \ \ \ \ \ \ \ \ \  \  \ \ \ \ \ \ \ \ \ \ \ \ \ \ \ \ \ \
\mbox{a.e. in}\ \mathbb{R}\backslash P_{2R_1},\end{array}
\right.
\end{equation}
where $u\in C([0,T],H^s(\mathbb{R}))$, $s>0$ large enough. From
Unique Continuation due to Zhang  \cite{zhang}, we conclude that
$u\equiv0$ a.e. in $P_{2R_1}$. Being $u\equiv0$ a.e. in
$\mathbb{R}\backslash
P_{R_1}$ we get $u\equiv0$ a.e. in $\mathbb{R}$, which is a contradiction.\\

Case (II): $u=0.$

Define,
\begin{equation}
\nu_k=||u_k||_{L^2([0,T];H^1(\mathbb{R}))}. \label{normalizeduk}
\end{equation}
Then, for $v_k=\displaystyle\frac{u_k}{\nu_k}$ we have

\begin{equation}
||v_k||_{L^2([0,T];H^1(\mathbb{R}))}=1, \ \ \ \forall\
k\in\mathbb{N}, \label{normavk}\end{equation} which implies that
there exists $v\in L^2([0,T];H^1(\mathbb{R}))$
\begin{eqnarray}\label{conv vk}
v_k \rightharpoonup v \hbox{ weakly in
}L^2([0,T];L^2(\mathbb{R}))~ \hbox{and}~
 v_{k,x} \rightharpoonup v_x
\hbox{ weakly in }L^2([0,T];L^2(\mathbb{R})).
\end{eqnarray}

\indent Moreover, $v_k$ satisfies the equation
\begin{equation}\label{vksequence}
\displaystyle v_{k,t}+u_k v_{k,x}+v_{k,xxx}+a(x)v_k=0,\ \ \ \ \ \
\ \mbox{in}\ \
\mathcal{D}'\left((R,+\infty)\times (0,T)\right)~(\forall R<
R_1).
\end{equation}
\\
\indent From $(\ref{blowUCP})$,
\begin{equation}\label{blowUCPvk}
\lim_{k \rightarrow +\infty}\displaystyle\frac{\int_0^T
||v_k(t)||_{L^2(P_{R_1})}^2 dt}{\int_0^T\int_{\mathbb{R}}\left(
a(x)\,v_k^2\right)dx\,dt}=+\infty.
\end{equation}

On the other hand, being $||v_k||_{L^2(0,T;L^2(P_{R_1}))}^2 \leq
||v_k||_{L^2(0,T;L^2(\mathbb{R}))}^2\leq 1$ for all
$k\in\mathbb{R}$ we obtain from $(\ref{blowUCPvk})$ that
\begin{equation}\label{normL2vk}
\displaystyle\lim_{k\rightarrow+\infty}\int_0^T\int_{\mathbb{R}}\left(
a(x)\,v_k^2\right)dx\,dt=0.
\end{equation}

Since $a(x)\geq\alpha_0>0$ for $x\in\{x\in\mathbb{R};\ x\geq R_1,\
R_1>0\}$ we obtain from $(\ref{normL2vk})$,
\begin{equation}\label{normL2vk1}
\displaystyle\lim_{k\rightarrow+\infty}\int_0^T\int_{\mathbb{R}\backslash
P_{R_1}}v_k^2dx\,dt=0.
\end{equation}

Thus,
\begin{equation}
v_k\rightarrow0\ \ \mbox{in}\ \ L^2([0,T];L^2(\mathbb{R}\backslash
P_{R_1})). \label{limit6}
\end{equation}

 Taking $(\ref{limit5})$,
$(\ref{utilde})$ (observe that $u=0$), $(\ref{vksequence})$, $(\ref{conv vk})$, $(\ref{normL2vk})$ and $(\ref{limit6})$ into account, we obtain,
\begin{equation}\label{kdvlimitvk}
\displaystyle\left\{\begin{array}{lll} v_{t}+v_{xxx}=0,\ \ \ \ \ \ \
\ \ \ \
\mbox{in}\ \ \mathcal{D}'((R,+\infty)\times(0,T)),\\\\
v=0,\ \ \ \ \ \ \ \ \ \ \ \ \  \ \ \ \ \ \ \ \ \mbox{a.e. in}\
\mathbb{R}\backslash {P_{2R_1}}.\end{array}\right.
\end{equation}

Therefore, from Holmgreen's Theorem  we conclude that $v\equiv 0$
in $(R,+\infty) \times [0,T]$, for all $R < R_1$, that is, $v\equiv 0$ in $\mathbb{R}$ which contradicts $(\ref{normavk})$.\\

\end{proof}

Using the result established in Lemma $\ref{lema6}$, we are able
to prove the result in Theorem $\ref{teo3}$. Indeed, taking
$(\ref{energyest})$ and $(\ref{estUCP})$ into account and making
use of the identity of the energy
\begin{eqnarray*}
E_0(T)- E_0(0) =-\int_0^T \int_{\mathbb{R}}a(x) u^2\,dx\,dt,
\end{eqnarray*}
we deduce that
\begin{eqnarray}
\int_0^T E_0(t)\,dt \leq C\,E_0(0),~\hbox{ for all }T > T_0,
\end{eqnarray}
for some positive constant $C>0$, which implies the exponential
stability.

\end{proof}
\section{Non-decay in $H^1-$norm.}
\indent In this section let us consider a particular case, when
equation $(\ref{kdv1})$ is fully damped, that is,

\begin{equation}\label{kdv3}
\displaystyle\left\{\begin{array}{lll} u_t+uu_x+u_{xxx}+\mu u=0,\ \
\ \ \ \ \
(x,t)\in\mathbb{R}\times[0,+\infty),\\\\
u(x,0)=u_0(x),\ \ \ \ \ \ \ \ \ \ \ \ \ \ \ \ \ \ \ \ \ \ \ \ \
x\in\mathbb{R},\end{array}\right.
\end{equation}
where $\mu>0$. It is straightforward to conclude that for $u_0\in
H^1(\mathbb{R})$ equation \ref{kdv3} possesses a unique solution
$u\in C([0,T];H^1(\mathbb{R}))$,  for all $T>0$. Moreover, the
solution $u$ satisfies a similar identity (see
$(\ref{conservada1})$), for all $t\geq0$, in this case given by,

\begin{equation}
\displaystyle\frac{d}{dt}\underbrace{\left\{||u||_{H^1}^2-\frac{1}{3}\int_{\mathbb{R}}u^3\right\}}_{E(t)}=
-\mu\underbrace{\left\{2||u||_{H^1}^2-\int_{\mathbb{R}}u^3\right\}}_{K(t)}
\label{conservada3}\end{equation}

Assuming that $K(t)\geq0$ occurs for all $t\geq0$ and considering
a $C^1-$functional $J:H^1(\mathbb{R})\rightarrow\mathbb{R}$
defined by
\begin{equation}{\label{functJu}}
J(w)=\displaystyle\frac{1}{2}||w||_{H^1}^2-\frac{1}{3}||w||_{L^3}^3;\
\ \ w\in H^1(\mathbb{R});
\end{equation}
we note that
\begin{equation}\label{EtJu}
E(t)\geq J(u).
\end{equation}

Moreover, since $K(t)\geq0$ for all $t\geq0$ we conclude from
$(\ref{conservada3})$
$$E'(t)\leq0,\ \ \ \forall\ t\in[0,+\infty),$$
and then
\begin{equation}\label{decreaEt}
E(t)\leq E(0), \ \ \ \ \forall\ t\in[0,+\infty).
\end{equation}

Let $B_1>0$ be the sharp Sobolev constant verifying ,
$$||u||_{L^3}\leq B_1||u||_{H^1},\ \ \ \ \forall\ u\in
H^1(\mathbb{R});$$ then
$$\displaystyle\frac{\frac{1}{3}||u||_{L^3}^3}{||u||_{H^1}^3}\leq \frac{B_1^3}{3},\ \ \ \ \forall\ u\in
H^1(\mathbb{R}).$$

The above inequality implies that

\begin{equation}
0<k_0=\displaystyle\sup_{v\in
H^1(\mathbb{R})\backslash\{0\}}\left(\frac{\frac{1}{3}||u||_{L^3}^3}{||u||_{H^1}^3}\right)\leq
\frac{B_1^3}{3}.\label{k0}\end{equation}

The functional $J$ defined in $(\ref{functJu})$ enable us to
define a function $f:[0,+\infty)\rightarrow\mathbb{R}$ given by
$f(\xi)=\displaystyle\frac{1}{2}\xi^2-k_0\xi^3$ (see figure
below), where $k_0$ was defined in $(\ref{k0})$. We note that
$f(||u||_{H^1})\leq J(u)$.\\
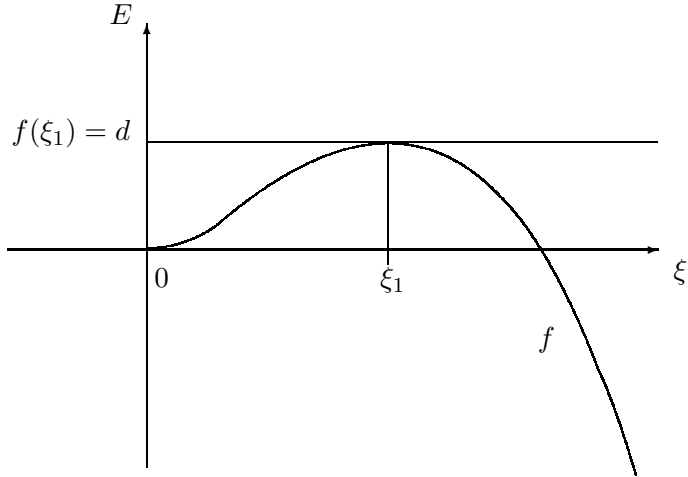
\begin{figure}[h]
\begin{picture}(8,7)(-0.5,-3)
\put(-1.85,0){\vector(1,0){8.65}}
\put(0,-2.9){\vector(0,1){5.9}}
\multiput(3.2,-0.2)(1,0){1}{\line(0,1){0.2}}
\put(0.1,-0.5){$0$}
\put(3.1,-0.5){$\xi_1$} \put(7.0,-0.4){$\xi$}
\put(-1.8,1.45){$f(\xi_1)= d$} \put(-0.5,3.0){$E$}
\put(0,1.432){\line(1,0){6.78}} \put(3.2,0){\line(0,1){1.44}}
\qbezier(0,0.01)(0.5,0.03)(0.9,0.3)
\qbezier(0.9,0.3)(3.2,2.3)(4.7,0.7) \qbezier(4.7,0.7)(5.0,0.4)(5.3,
-0.1) \qbezier(5.3,-0.1)(5.7,-0.8)(6.0,-1.6)
\qbezier(6.0,-1.6)(6.2,-2.0)(6.5,-3.0)

\put(5.2,-1.3){$f$}
\end{picture}
\medskip
\caption{ Graphic of function $f(\xi)= \frac12 \xi^2
-\frac{1}{k_0}\xi^3$ . }
\end{figure}

Since $f$ is a differentiable function in $\mathbb{R}_{+}$ their
critical points occur in $\xi_0=0$ and
$\xi_1=\displaystyle\frac{1}{3k_0}$. The last point is, in fact,
the absolute maximum of the function $f$. Moreover, we define
$d:=f(\xi_1)=\displaystyle\frac{1}{6}\xi_1^2$.

Next, from $(\ref{EtJu})$, $(\ref{decreaEt})$ and $(\ref{k0})$ we
obtain the estimate,
\begin{equation}
E(t)\geq J(u)\geq f(||u||_{H^1}). \label{EtJuf}\end{equation}

We have the following result due to Vitillaro (see
\cite{vitillaro}),

\begin{lema}\label{Vitilaro}
Let $u$ be a solution that exists on the interval $[0,T]$, for all
$T>0$, according to Theorem $\ref{teo2}$ for $a(x)\equiv\mu>0$.
Then, if $||u_0||_{H^1}>\xi_1$ and $E(0)<d$, then
$$||u(t)||_{H^1}\geq\xi_2,$$
for some $\xi_2>\xi_1$ and all $t\geq0$. Moreover,
$$||u||_{L^3}\geq
k_0^{1/3}\xi_2.$$
\end{lema}
\begin{proof} Indeed, since $f$ is strictly increasing for $0<\xi<\xi_1$
and strictly decreasing for $\xi>\xi_1$, $f(\xi_1)=d$,
$f(\xi)\rightarrow-\infty$ when $\xi\rightarrow+\infty$ and
$d>E(0)>f(||u_0||_{H^1})\geq f(0)=0$, there are $\xi_2'<\xi_1<\xi_2$
such that,

\begin{equation}\label{E0fxi}
f(\xi_2)=f(\xi_2')=E(0).
\end{equation}
\indent On the other hand, being $E(t)$ non-increasing, we have
\begin{equation}\label{EtE0}
E(t)\leq E(0),\ \ \ \ \forall\ t\geq0.
\end{equation}

From $(\ref{EtJuf})$ and $(\ref{E0fxi})$ we deduce

\begin{equation}\label{E0fxi2}
f(||u_0||_{H^1})\leq E(0)=f(\xi_2).
\end{equation}

Since $||u_0||_{H^1},\xi_2\in(\xi_1,+\infty)$ we conclude from
$(\ref{E0fxi2})$ that
\begin{equation}
||u_0||_{H^1}\geq\xi_2. \label{u0xi2}
\end{equation}

The next step is to prove that
\begin{equation}\label{resulteo}
||u(t)||_{H^1}\geq\xi_2,\ \ \ \forall t\geq0.
\end{equation}

In fact, we argue by contradiction. If $(\ref{resulteo})$ does not
occur, then there is $t^{\star}\in(0,+\infty)$ such that,
\begin{equation}\label{resulteocont}
||u(t^{\star})||_{H^1}\leq\xi_2.
\end{equation}

It is necessary to consider two cases:\\
(i) If $||u(t^{\star})||_{H^1}>\xi_1$, then from $(\ref{EtJuf})$,
$(\ref{E0fxi2})$ and $(\ref{resulteocont})$ we obtain,
$$E(t^{\star})\geq f(||u(t^{\star})||_{H^1})>f(\xi_2)=E(0),$$
which contradicts $(\ref{decreaEt})$.

(ii) If we consider $||u(t^{\star})||_{H^1}\leq\xi_1$, observing
$(\ref{u0xi2})$, we guarantee the existence of $\bar{\xi}$ which
satisfies
\begin{equation}\label{contiuu0}
||u(t^{\star})||_{H^1}\leq\xi_1<\bar{\xi}<\xi_2\leq||u_0||_{H^1}.
\end{equation}

Consequently, from the continuity of $||u(\cdot)||_{H^1}$ as a
function of $t\geq0$, there is a $\bar{t}\in(0,t^{\star})$ verifying
$$||u(\bar{t})||_{H^1}=\bar{\xi}.$$

From the last equality, and $(\ref{EtJuf})$, $(\ref{E0fxi2})$ and
$(\ref{contiuu0})$,
$$E(\bar{t})\geq
f(||u(\bar{t})||_{H^1})=f(\bar{\xi})>f(\xi_1)=E(0),$$ which
contradicts $(\ref{decreaEt})$. Since $E(t)\leq E(0)$ we get
$$||u||_{L^3}\geq k_0^{1/3}\xi_2.$$

This proves the Lemma.

\end{proof}

\begin{obs}\label{obs 2.3}
From Lemma $\ref{Vitilaro}$ and Sobolev's embedding
$H^1(\mathbb{R})\hookrightarrow L^3(\mathbb{R})$ we guarantee that

\begin{equation}\label{sobolevk0}
||u||_{H^1}\geq C||u||_{L^3}\geq k_0^{1/3}\xi_2,
\end{equation}
for some positive constant $C>0$ and initial data $u_0$ satisfying
$||u_0||_{H^1}>\xi_1=\displaystyle\frac{1}{3k_0}$. The inequality
$(\ref{sobolevk0})$ suggests that the $H^1-$norm of the solution
$u$ related to the fully damped Korteweg -de Vries equation
$(\ref{kdv3})$, can not have a decay property even being the
energy $E(t)$ non-increasing.
\end{obs}

\medskip

{\bf Acknowledgements.} The authors would like to express their gratitude to the anonymous referees for giving constructive suggestions which allow to improve considerably the initial version of the present manuscript.


\begin{thebibliography}{99}

\bibitem{Alabau} F. Alabau-Boussouira, Convexity and weighted integral
inequalities for energy decay rates of nonlinear dissipative
hyperbolic systems, {\em Appl. Math. Optim.} 51(1), (2005),
61-105.


\bibitem{albert1}  J.P. Albert, \textit{Positivity properties and stability of solitary-wave
solutions of model eqautions for long waves}, Comm PDE, 17 (1992)
pp. 1--22.

\bibitem{natali} J. Angulo, and F. Natali, \textit{Positivity Properties of the Fourier
Transform and the Stability of Periodic Travelling-Wave Solutions},
SIAM J. Math. Anal., 40 (2008) pp. 1123-1151.

\bibitem{angulo} J. Angulo, J.L. Bona, M. and Scialom, \textit{Stability of cnoidal waves},
Advances in Differential Equations 11 pp. 1321--1374 (2006).

\bibitem{benjamin1} T.B. Benjamin, \textit{The stability of solitary waves}, Proc. Roy. Soc.
London Ser. A 338 (1972), pp. 153--183.

\bibitem{Bisognin} E. Bisognin, V. Bisognin, G. Perla Menzala, \textit{Exponential stabilization of a coupled system
of Korteweg-de Vries equations with localized damping}. Adv.
Differential Equations 8 (2003), no. 4, pp. 443--469.

\bibitem{bona1}  J.L. Bona, \textit{On the stability theory of solitary waves}, Proc Roy. Soc.
London Ser. A 344 (1975), pp. 363--374.

\bibitem{Cavalcanti} M.M. Cavalcanti, V.N. Domingos Cavalcanti and P. Martinez,
\textit{Existence and decay rate estimates for the wave equation
with nonlinear boundary damping and source term. J. Differential
Equations} 203 (2004), no. 1, pp. 119--158.

\bibitem{Cavalcanti2} M.M. Cavalcanti, V.N. Domingos Cavalcanti and I. Lasiecka, \textit{Well-posedness and optimal decay
rates for the wave equation with nonlinear boundary damping---source
interaction}. J. Differential Equations 236 (2007), no. 2, pp.
407--459.

\bibitem{Cavalcanti3} C. Alves and M. M. Cavalcanti. On existence, uniform decay rates and blow up for solutions of the
2-D wave equation with exponential source. {\em Calculus of
Variations on PDE} 34 (2009), no. 3, 377--411.

\bibitem{Gjidaglia} J.M. Ghidaglia, \textit{A note on the strong convergence towards attractors of damped forced KdV
equations}. J. Differential Equations 110 (1994), no. 2, pp.
356--359.

\bibitem{Felipe} F. Linares and A. Pazoto, \textit{On the exponential decay of the
critical generalized Korteweg-de Vries equation with localized
damping}. Proc. Amer. Math. Soc. 135 (2007), no. 5,  pp. 1515--1522.

\bibitem{grimshaw} R. Grimshaw, and H. Mitsudera, \textit{Slowly varying solitary wave
solutions of the pertubed Korteweg-de Vries equation revisted},
Stud. Appl. Math. 90 (1993), no. 1, pp. 75--86

\bibitem{lions1}  J.L. Lions, \textit{Quelques m\'ethodes de re\'esolution des probl\`emes
aux limites non lin\'eaires}, Dunod-Paris (1969).
\bibitem{KPV93} C.E. Kenig, G. Ponce, and L. Vega, \textit{Well-posedness and
scattering results
for the generalized Korteweg-de Vries equation via the contraction
principle}, Comm. Pure Appl. Math. 46 (1993), pp. 527--560

\bibitem{KPV91}  C.E. Kenig, G. Ponce, and L. Vega, \textit{Well-posedness of the
initial value problem for the Korteweg-de Vries equation}, J. Amer.
Math. Soc. 4 (1991), no 2, pp. 323--347.

\bibitem{Komornik} V. Komornik,  \textit{Exact Controllability and Stabilization.
The Multiplier Method},  Mason-John Wiley, Paris, 1994.

\bibitem{Lasiecka-Tataru} I. Lasiecka and D. Tataru, \textit{Uniform boundary
stabilization of semilinear wave equation with nonlinear boundary
damping}, Differential and Integral Equations, 6 (1993), pp.
507--533.

\bibitem{KDV-weak} C.P. Massarolo, G. Perla Menzala and A.F. Pazoto,  \textit{On the uniform
decay for the Korteweg-de Vries equation with weak damping}. Math.
Methods Appl. Sci. 30 (2007), no. 12, pp. 1419--1435.

\bibitem{Pazoto} A.F. Pazoto, \textit{Unique continuation and decay for the
Korteweg-de Vries equation with localized damping}. ESAIM Control
Optim. Calc. Var. 11 (2005), no. 3, pp. 473--486.

\bibitem{Perla} G.P. Menzala, C.F. Vasconcellos and  E. Zuazua, \textit{Stabilization of the Korteweg-de Vries equation
with localized damping}. (English) J. Q. Appl. Math. 60, No. 1,
(2002) pp. 111--129 .


\bibitem{pazy} A. Pazy, \textit{Semigroups of linear operators and applications to partial
differential equations}, Springer, 3rd edition, (1983).

\bibitem{rosierzhang} L. Rosier and B.Y. Zhang, \textit{Global stabilization of the generalized
korteweg–de vries equation posed on a finite domain}, SIAM J.
Control Optim. 3 (2006) pp. 927--956.

\bibitem{vitillaro} E. Vitillaro, \textit{Some new results on global nonexistence and blow-up for evolution problems
with positive initial energy}, Rend. Inst. Mat. Univ. Trieste, 31
(2000), pp. 375--395.

\bibitem{zhang} B.Y. Zhang, \textit{Unique continuation for the Korteweg-de Vries
Equation}, SIAM J. Math. Anal. 23 (1992), pp. 55--71.

\bibitem{W1} M.I. Weinstein, \textit{Liapunov stability of ground
states of nonlinear dispersive evolution equations}. Comm. Pure
Appl. Math.,  v. 39, (1986), pp. 51--68.



\end{thebibliography}
\end{document}